\documentclass[11pt,reqno]{amsart}
\usepackage{amsmath,amsthm,amscd,amsfonts,amssymb,color}
\usepackage{cite}
\usepackage{graphicx}
\graphicspath{ {./images/} }
\usepackage[mathscr]{eucal}
\usepackage[bookmarksnumbered,colorlinks,plainpages]{hyperref}
\usepackage[margin=2.25in]{geometry}
\newtheorem{theorem}{Theorem}[section]
\newtheorem{lemma}[theorem]{Lemma}
\newtheorem*{Acknowledgement}{Acknowledgement}

\theoremstyle{definition}

\newtheorem{example}[theorem]{Example}
\theoremstyle{remark}

\numberwithin{equation}{section}
\def\DJ{\leavevmode\setbox0=\hbox{D}\kern0pt\rlap
 {\kern.04em\raise.188\ht0\hbox{-}}D}

\begin{document}

\title[Solution to a pair of matrix equations]{On new existence of a unique common solution to a pair of non-linear matrix equations}

\author[H.\ Garai, L.K. \ Dey, W. Sintunavarat,  S. \ Som, S. \ Raha]
{Hiranmoy Garai$^{1}$, Lakshmi Kanta Dey$^{2}$, Wutiphol Sintunavarat$^{3}$,  Sumit Som$^{4}$, Sayandeepa Raha$^{5}$}

\thanks{Corresponding author: W. Sintunavarat (E-mail: wutiphol@mathstat.sci.tu.ac.th)}
\address{{$^{1}$\,} Hiranmoy Garai,
                    Department of Mathematics,
                    National Institute of Technology
                    Durgapur, India.}
                    \email{hiran.garai24@gmail.com}
\address{{$^{2}$\,} Lakshmi Kanta Dey,
                    Department of Mathematics,
                    National Institute of Technology
                    Durgapur, India.}
                    \email{lakshmikdey@yahoo.co.in}                    
\address{{$^{3}$\,} Wutiphol Sintunavarat, Department of Mathematics and Statistics,    					        Faculty of Science and Technology, 
                    Thammasat University Rangsit Center, Pathum Thani 12120, Thailand.}
                    \email{wutiphol@mathstat.sci.tu.ac.th}
\address{{$^{4}$\,} Sumit Som,
                    Department of Mathematics,
                    National Institute of Technology
                    Durgapur, India}
                    \email{somkakdwip@gmail.com}
\address{{$^{5}$\,} Sayandeepa Raha,
                    Department of Biotechnology,
                    National Institute of Technology
                    Durgapur, India}
                    \email{rahasayandeepa@gmail.com}

\begin{abstract}
The main goal of this article is to study the existence of a unique positive definite common solution to a pair of matrix equations of the form \begin{eqnarray*}
X^r=Q_1 + \displaystyle \sum_{i=1}^{m} {A_i}^*F(X)A_i \mbox{ and } X^s=Q_2 + \displaystyle \sum_{i=1}^{m} {A_i}^*G(X)A_i
\end{eqnarray*}
where $Q_1,Q_2\in P(n)$, $A_i\in M(n)$ and $F,G:P(n)\to P(n)$ are certain functions and $r,s>1$. In order to achieve our target, we take the help of elegant properties of Thompson metric on the set of all $n \times n$ Hermitian positive definite matrices. To proceed this, we first derive a common fixed point result for a pair of mappings utilizing a certain class of control functions in a metric space. Then, we obtain some sufficient conditions to assure a unique positive definite common solution to the said  equations. Finally, to validate  our results, we  provide a couple of numerical examples with diagrammatic representations of  the convergence behaviour of iterative sequences.
\end{abstract}
 
\maketitle

\setcounter{page}{1}

\section{Introduction}

The metrical fixed point theory is one of essential and important tools for solving various types equations arising in different mathematical problems.
Throughout the last hundred years, many researchers have studied metrical fixed point theory by investigating several interesting existence (and uniqueness) fixed point results for different types of contractions, and they have utilized those results for solving different equations arising in  optimization theory, approximation theory, learning theory,  variational inequality and many more.
Based on the  impact of several real-world problems, the metrical fixed point theory becomes one of the foremost subject in modern research. 
The topic of common fixed points for a pair or a family of mappings in metric  spaces is of great interest and play important roles in many applications. Throughout the last decade, many mathematicians develop the study of common fixed points by establishing several sufficient conditions involving compatibility, weak compatibility, commutativity, continuity, the E.A. property, the common limit in the range property and numerous others for existence (and uniqueness) of common fixed point of pair and family of mappings. Motivated by some important and interesting fixed point and common fixed point results, we enrich the common fixed point theory in terms of a control function in Section 2 of this  work. 

On the another side, there is a huge application of matrix equations in various types  of problems in engineering, computer sciences, stability analysis and many more,  (see \cite{W1,TW} and the references therein). Throughout the last few years, many researchers have achieved different kinds of sufficient conditions for ensuring general solution(s) of different types of matrix equations, see \cite {RR,ER,BS,  HBP, HBK, SS,SS1, SS2}. Besides these, very recently, Garai and Dey \cite{GD} noticed that there are some problems in non-interacting control theory, where two different matrix equations arise and these types of problems require common solution to certain pair of matrix equations.  Garai and Dey \cite{GD} were the first to handle this situation. More precisely, they obtained some adequate conditions to ensure the existence of unique common positive definite solution to the following equations
\begin{equation}\label{me1}
X=Q_1 + \displaystyle \sum_{i=1}^{m} {A_i}^*F(X)A_i, 
\end{equation}
\begin{equation}\label{me01}
X=Q_2 + \displaystyle \sum_{i=1}^{m} {A_i}^*G(X)A_i, 
\end{equation}
where $Q_1\ Q_2\in P(n)$ (set of all $n\times n$ Hermitian positive definite matrices), $A_i\in M(n)$ (set of all $n \times n$ matrices), $A_i^*$ denotes the conjugate transpose of $A_i$, $F, G$ are mappings from $H(n)$ (set of all $n \times n$ Hermitian matrices) into itself. But the above type of equations are particular case of the general matrix equation
\begin{equation}\label{me001}
X^r=Q + \displaystyle \sum_{i=1}^{m} {A_i}^*F(X)A_i, 
\end{equation}
where $Q,\ A_i, F$ are as above and $r\geq 1$. But, in reality, most of the time we have to handle the general matrix equations.

Motivated by the above scenario, our major goal of the present paper is to  focus on two pairs of matrix equations. The first pair is as follows:
\begin{align}
X^s=Q_1 + \displaystyle \sum_{i=1}^{m} {A_i}^*F(X)A_i \label{add01},\\
X^s=Q_2 + \displaystyle \sum_{i=1}^{m} {A_i}^*G(X)A_i \label{add02}
\end{align}
where $Q_1,~Q_2\in P(n)$, $A_i's$ are arbitrary nonsingular complex matrices of order $n$, $F,~G$ are two functions from $P(n)$ to $P(n)$ and $s>1$ is a real number, meanwhile, the second pair is as follows:
\begin{align}
X^r=\displaystyle \sum_{i=1}^{m} {A_i}^*F(X)A_i, \label{add03}\\
X^s=\displaystyle \sum_{i=1}^{m} {A_i}^*G(X)A_i \label{add04}
\end{align}
where $F,G$ are same as the first system, $A_i$'s are arbitrary $n \times n$ unitary matrices, and $r,s>1$ are real numbers. 

In order to solve the two mentioned systems, we first obtain a common fixed point result using  a special type of family of functions in complete metric spaces. Based on this common fixed point result, we present some adequate conditions for the existence of unique common positive definite solution of both the systems. Finally, we authenticate our obtained sufficient conditions by presenting some numerical examples along with clear diagrammatic representations of  the convergence behaviour of iterative sequences.

\section{The common fixed point results}
In this section, we present a new common fixed point result for a pair of mappings.  Before investigating such a common fixed point result, first, we introduce a collection of some control functions.
 
For a fixed real number $\alpha \in [0,1)$, we use the symbol $\Psi_{\alpha}$ to denote the collection of all mappings $\psi:\mathbb{R_+}^3 \to[0,\infty)$ satisfying the following conditions:
\begin{enumerate}
\item[$(\psi')$] $\psi$ is continuous;
\item[$(\psi'')$] if $b \leq \psi(a,a,b)$ or $b \leq \psi(b,a,a)$ or $b \leq \psi(a,b,a)$, then $b\leq \alpha a$.
\end{enumerate}
Some examples of mappings belonging to the class $\Psi_{\alpha}$ are given by the following:
\begin{enumerate}
\item[$(i)$] a mapping $\psi:\mathbb{R_+}^3 \to[0,\infty)$ which is defined by 
$$\psi(a,b,c)=\alpha a$$ 
for all $a,,b,c\in \mathbb{R_+}$, where $\alpha \in [0,1)$;
\item[$(ii)$] a mapping $\psi:\mathbb{R_+}^3 \to[0,\infty)$ which is defined by 
$$\psi(a,b,c)=Ma+Nb+Oc$$ 
for all $a,b,c \in \mathbb{R_+}$, where $M,N,O$ are nonnegative real numbers with $M+N+O<1$;
\item[$(iii)$] a mapping $\psi:\mathbb{R_+}^3 \to[0,\infty)$ which is defined by 
$$\psi(a,b,c)=\alpha \max\{a,b,c\}$$ 
for all $a,b,c \in \mathbb{R_+}$, where $\alpha \in [0,1)$.
\end{enumerate}

\begin{theorem}\label{th1}
Let $(C,d)$ be a complete metric space and $T_1,T_2:C\to C$ be two mappings satisfying the following condition:
\begin{equation}\label{e1}
d(T_1(x),T_2(y)) \leq \psi\left(d(x,y),d(x,T_1(x)),d(y,T_2(y))\right)
\end{equation}
for all $x,y \in C$, where $\psi \in \Psi_{\alpha}$. Then the following assertions hold:
\begin{enumerate}
\item[$(i)$] $T_1$ and $T_2$ have a unique common fixed point;
\item[$(ii)$] for any $u_0\in C$, the sequence $\{u_n\}$ converges to that common fixed point, where $u_{2k}=T_2(u_{2k-1})$ and $u_{2k+1}=T_1(u_{2k})$ for all $k \in \mathbb{N},$ and the error estimation is given by
$$d(u_n,z)\leq \frac{\alpha^{n-1}}{1-\alpha}d(u_0,u_1),$$ 
for all $n \in \mathbb{N}$, where $z$ is the unique common fixed point of $T_1$ and $T_2$.
\end{enumerate}
\end{theorem}

\begin{proof}
Let $u_0\in C$ be arbitrary but fixed and consider the sequence $\{u_n\}$ which is defined by taking 
\begin{center}
	$u_{2k}=T_2(u_{2k-1})$ and $u_{2k+1}=T_1(u_{2k})$ 
\end{center}
for all $k \in \mathbb{N}$. First, we assume that $n$ is even. Then putting $x=u_n$ and $y=u_{n+1}$ in  \eqref{e1}, we get
\begin{eqnarray}
	d(u_{n+1},u_{n+2}) 
		&=& d(T_1(u_n),T_2(u_{n+1})) \nonumber \\
		&\leq& \psi(d(u_n,u_{n+1}),d(u_n,T_1(u_n)),d(u_{n+1},T_2(u_{n+1}))) \nonumber \\
		&=& \psi(d(u_n,u_{n+1}),d(u_n,u_{n+1}),d(u_{n+1},u_{n+2})). \nonumber
\end{eqnarray}
By using the property of $\psi$, we get $$d(u_{n+1},u_{n+2})\leq \alpha d(u_n,u_{n+1}),$$ 
where $\alpha$ is a constant lying in $[0,1)$. 
Next, we assume that $n$ is odd. Then the proceeding in a similar manner as above, we can show that $$d(u_{n+1},u_{n+2})\leq \alpha d(u_n,u_{n+1}).$$
Therefore, we have
\begin{eqnarray}
d(u_{n+1},u_{n+2})
&\leq& \alpha d(u_n,u_{n+1}) \nonumber \\
&\leq& \alpha^2 d(u_{n-1},u_{n}) \nonumber \\
& \vdots& \nonumber  \\
&\leq& \alpha^n d(u_{0},u_{1}) \nonumber 
\end{eqnarray}
for any $n \in \mathbb{N}$. 
Since $\alpha\in [0,1)$, it follows that the infinite series $\sum_{n=1}^{\infty} d(u_n,u_{n+1})$ is convergent. Thus, $\{u_n\}$ is a Cauchy sequence in $C$ and so there exists $z\in C$ such that $u_n\to z$ as $n\to \infty$. So the subsequences $\{u_{2k}\}$ and $\{u_{2k+1}\}$ converge to $z$.

Now, for any $k \in \mathbb{N}$, we have
\begin{eqnarray}
d(u_{2k+1},T_2(z))
&=& d(T_1(u_{2k}),T_2(z))	 \nonumber \\
&\leq &  \psi(d(u_{2k},z),d(u_{2k},T_1(u_{2k})),d(z,T_2(z))) \nonumber \\
&=& \psi(d(u_{2k},z),d(u_{2k},u_{2k+1}),d(z,T_2(z))). \nonumber \\
\end{eqnarray}
By taking the limit as $k \to \infty$ in the above inequality, we get
$$d(z,T_2(z)) \leq \psi(0,0,d(z,T_2(z)))$$
and so $d(z,T_2(z)) \leq \alpha \cdot 0$. This implies that $z=T_2(z)$. Similarly, we can show that $z=T_1(z)$. So $z$ is a common fixed point of $T_1$ and $T_2$. 

Next, we will show the uniqueness of the common fixed point of $T_1$ and $T_2$. For this, let $z_1$ be another common fixed point of $T_1$ and $T_2$. Then putting $x=z$, $y=z_1$ in \eqref{e1}, we get 
\begin{eqnarray}
	d(T_1(z),T_2(z_1)) 
	&\leq & \psi(d(z,z_1),d(z,T_1(z)),d(z_1,T_2(z_1))) \nonumber \\
	&=& \psi(d(z,z_1),0,0) \nonumber 
\end{eqnarray}
and so $d(z,z_1) \leq \alpha \cdot 0$. It yields that $z=z_1$. Therefore, $z$ is the unique common fixed point of $T_1$ and $T_2$. This proves $(i)$.

Finally, we will show that the assertion $(i)$ holds.
Now for any $n,m \in \mathbb{N}$ with $n<m$, we have
\begin{eqnarray}
d(u_n,u_m)&\leq& d(u_n,u_{n+1})+d(u_{n+1}.u_{n+2})+ \hdots + d(u_{m-1},u_m)\nonumber\\
&\leq& (\alpha^{n-1}+\alpha^{n}+\hdots+\alpha^{m-2})d(u_0,u_1)\nonumber\\
&=&\frac{\alpha^{n-1}-\alpha^m}{1-\alpha} d(u_0,u_1). \nonumber
\end{eqnarray}
Keeping $n$ fixed and letting $m \to \infty$ in the above equation, we get 
$$d(u_n,z)\leq \frac{\alpha^{n-1}}{1-\alpha} d(u_0,u_1)$$ 
for all $n\in \mathbb{N}$. 
This proves $(ii)$.
\end{proof}
\section{Common solution to matrix equations}


Right through this section, we consider matrices over the set of complex numbers.  We use the notation $I_n$ as the $n\times n$ identity matrix. The Thompson metric $d$ on the set $P(n)$ is defined by 
$$d(A,B)= \max\{\log W(A/B), \log W(B/A)\}$$ 
for all $A,B \in P(n)$, where $$W(A/B):=\inf\{\delta>0:A\leq \delta B\}=\lambda^{+}(B^{-\frac{1}{2}}AB^{-\frac{1}{2}})$$ 
is the maximum eigenvalue of the matrix $B^{-\frac{1}{2}}AB^{-\frac{1}{2}}$. It is well known that the set $P(n)$ is complete with respect to the Thompson metric $d$. The following properties involving  the Thompson metric are necessary to recall for the sake of our  developments.

\begin{lemma}[\cite{L,T}]
	Let $d$ be the Thompson metric on the set $P(n)$. Then
	\begin{itemize}
		\item[$(i)$] $d(A,B) = d(A^{-1},B^{-1})= d(MAM^{*},MBM^{*})$ for all $A,B \in P(n)$ and non-singular matrix $M$.
		\item[$(ii)$] $d(A^{r},B^{r})\leq \mid r\mid d(A,B)$ for all $A,B \in P(n)$ and $r \in [-1,1]$.
		\item[$(iii)$] $d(A+B,C+D)\leq \max\{d(A,C),d(B,D)\}$ for all $A,B,C,D \in P(n)$. In particular, $d(A+B,A+D)\leq d(B,D)$.
	\end{itemize}
\end{lemma}

\begin{theorem}\label{tt1}
Consider the  pair of  matrix equations  as follows:
\begin{align}
X^s=Q_1 + \displaystyle \sum_{i=1}^{m} {A_i}^*F(X)A_i \label{te1},\\
X^s=Q_2 + \displaystyle \sum_{i=1}^{m} {A_i}^*G(X)A_i \label{te20}
\end{align}
where $Q_1,~Q_2$ are two $n\times n$ Hermitian positive definite matrices, $A_i's$ are arbitrary nonsingular complex matrices of order $n$, $F,~G$ are two functions from $P(n)$ to $P(n)$ and $s>1$, $s\in \mathbb{R}$. Assume that there is a  $a\in [0,\infty)$ such that for any $X\in P(n)$ with $d(X,I_n) \leq e^a$, where $d$ is the Thompson metric on $P(n)$, the following conditions hold:
\begin{enumerate}
\item [$(A)$] $\lambda_{max} \left(Q_1^{-\frac{1}{2}}Q_2Q_1^{-\frac{1}{2}}\right) \leq \lambda_{max} \left(F(X)^{-\frac{1}{2}}G(Y) F(X)^{-\frac{1}{2}}\right)$ and $\lambda_{max} \left(Q_2^{-\frac{1}{2}}Q_1Q_2^{-\frac{1}{2}}\right) \leq \lambda_{max} \left(G(Y)^{-\frac{1}{2}}F(X) G(Y)^{-\frac{1}{2}}\right)$ ;
\item [$(B)$]  $\lambda_{max} \left(F(X)^{-\frac{1}{2}}G(Y) F(X)^{-\frac{1}{2}}\right) \leq \left( \lambda_{max} \left(X^{-\frac{1}{2}}YX^{-\frac{1}{2}}\right)\right)^l$ and \\ $\lambda_{max} \left(G(Y)^{-\frac{1}{2}}F(X) G(Y)^{-\frac{1}{2}}\right) \leq  \left(\lambda_{max} \left(Y^{-\frac{1}{2}}X Y^{-\frac{1}{2}}\right)\right)^l$, where $l>0$ is a real number such that $l<s$;
\item [$(C)$]  $\lambda_{max} \left( I_n^{-\frac{1}{2}}\left(Q_1+ \sum_{i=1}^{m} {A_i}^*F(X)A_i\right)^\frac{1}{s} I_n^{-\frac{1}{2}}\right),$ $\lambda_{max} \left( \left(\left(Q_1+ \sum_{i=1}^{m} {A_i}^*F(X)A_i\right)^\frac{1}{s} \right)^{-\frac{1}{4}}\right),$  $\lambda_{max} \left( I_n^{-\frac{1}{2}}\left(Q_2+ \sum_{i=1}^{m} {A_i}^*G(X)A_i\right)^\frac{1}{s} I_n^{-\frac{1}{2}}\right)$, $\lambda_{max} \left( \left(\left(Q_2+ \sum_{i=1}^{m} {A_i}^*G(X)A_i\right)^\frac{1}{s} \right)^{-\frac{1}{4}}\right)$ $\leq e^a$.
\end{enumerate}
Then the following assertions hold:
\begin{enumerate}
\item[$(i)$] the pair of matrix equations \eqref{te1} and \eqref{te20} posses a unique common solution $\bar{X}$ such that $\bar{X}$ is positive definite and $d(\bar{X},I_n)\leq a$;
\item [$(ii)$] the unique solution is provided by the limit of the iterative sequence $\{X_k\}$, where 
\begin{equation}
	X_k =\left\{
	\begin{tabular}{ccc}
		$\left(Q_1 + \displaystyle \sum_{i=1}^{m} {A_i}^*F(X_{k-1})A_i\right)^{\frac{1}{s}}$ & if & $k$ is odd, \\
		$\left(Q_2 + \displaystyle \sum_{i=1}^{m} {A_i}^*G(X_{k-1})A_i\right)^{\frac{1}{s}}$ & if & $k$ is even, 
	\end{tabular}
	\right.
\end{equation}
and 
$X_0 \in P(n)$ is any element with $d(X_0,I_n) \leq a$, and the error estimation is given by $$d(X_k,\bar{X}) \leq \frac{\alpha^{k-1}}{1-\alpha}d(X_0,X_1)$$ 
for all $k\in \mathbb{N}$, where $\alpha:=\frac{l}{s}$.
\end{enumerate}
\end{theorem}

\begin{proof}
Consider the set $\mathcal{C}=\{X\in P(n):d(X,I_n)\leq a\}$ 
and the Thompson metric $d$ on $\mathcal{C}$. 
Then $\mathcal{C}$ is complete with respect to the Thompson metric $d$.  Let us now consider a pair of mappings $T_1,T_2 :\mathcal{C} \to P(n)$ by 
\begin{equation}\label{te3}
T_1(X)=\left(Q_1 + \displaystyle \sum_{i=1}^{m} {A_i}^*F(X)A_i\right)^{\frac{1}{s}}
\end{equation}
\begin{equation}\label{te4}
T_2(X)=\left(Q_2 + \displaystyle \sum_{i=1}^{m} {A_i}^*G(X)A_i\right)^{\frac{1}{s}}
\end{equation}
for all $X \in \mathcal{C}$. 
From the condition $(C)$, we get 
$$d(T_1(X),I_n)=d\left(\left(Q_1+ \sum_{i=1}^{m} {A_i}^*F(X)A_i\right)^\frac{1}{s},I_n\right)\leq a$$ 
and 
$$d(T_2(X),I_n)=d\left(\left(Q_2+ \sum_{i=1}^{m} {A_i}^*G(X)A_i\right)^\frac{1}{s},I_n\right)\leq a.$$  
This yields that $T_1,T_2$ are two self-mappings on $\mathcal{C}$.

Next, we will show that $T_1$ and $T_2$ satisfy the contractive condition in Theorem  \ref{th1}. Let $X, Y\in \mathcal{C}$ be arbitrary. Then we have
\begin{eqnarray}
d\left(T_1(X),T_2(Y)\right)
&=&d\left(\left(Q_1 + \displaystyle \sum_{i=1}^{m} {A_i}^*F(X)A_i\right)^{\frac{1}{s}},\left(Q_2 + \displaystyle \sum_{i=1}^{m} {A_i}^*G(Y)A_i\right)^{\frac{1}{s}}\right)\nonumber\\
&\leq& \frac{1}{s} d\left(\left(Q_1 + \displaystyle \sum_{i=1}^{m} {A_i}^*F(X)A_i\right),\left(Q_2 + \displaystyle \sum_{i=1}^{m} {A_i}^*G(Y)A_i\right)\right)\nonumber\\
&\leq& \frac{1}{s} \max \left\{d(Q_1,Q_2),d\left(\displaystyle \sum_{i=1}^{m} {A_i}^*F(X)A_i,\sum_{i=1}^{m} {A_i}^*G(Y)A_i\right)\right\}. \nonumber \\
\label{te30}	
\end{eqnarray}
Now, we have 
\begin{align*}
&d\left(\displaystyle \sum_{i=m-1}^{m} {A_i}^*F(X)A_i,\sum_{i=m-1}^{m} {A_i}^*G(Y)A_i\right)\\
&\leq \max\left\{d\left(A_{m-1}^*F(X)A_{m-1},A_{m-1}^*G(Y)A_{m-1}\right),d\left(A_{m}^*F(X)A_{m},A_{m}^*G(Y)A_{m}\right)\right\}\\
&= \max\left\{d(F(X),G(Y)),d(F(X),G(Y))\right\}\\
&=d(F(X),G(Y)).
\end{align*}
Again, we have
\begin{align*}
&d\left(\displaystyle \sum_{i=m-2}^{m} {A_i}^*F(X)A_i,\sum_{i=m-2}^{m} {A_i}^*G(Y)A_i\right)\\
&\leq \max \left\{d\left(A_{m-2}^*F(X)A_{m-2},A_{m-2}^*G(Y)A_{m-2}\right),d\left(\displaystyle \sum_{i=m-1}^{m} {A_i}^*F(X)A_i,\sum_{i=m-1}^{m} {A_i}^*G(Y)A_i\right) \right\}\\
&\leq \max\left\{d(F(X),G(Y)),d(F(X),G(Y))\right\}\\
&=d(F(X),G(Y)).
\end{align*}
Continuing in this way, we can show that 
\begin{equation}\label{te40}
d\left(\displaystyle \sum_{i=1}^{m} {A_i}^*F(X)A_i,\sum_{i=1}^{m} {A_i}^*G(Y)A_i\right) \leq d(F(X),G(Y)).
\end{equation}
Using the equation \eqref{te40} in the equation \eqref{te30}, we get
\begin{equation}\label{te5}
d(T_1(X),T_2(Y))\leq \frac{1}{s} \max\left\{d(Q_1,Q_2),d(F(X),G(Y))\right\}.
\end{equation}
From the condition $(A)$, we get 
$$\log W(Q_2/Q_1)\leq \log W(G(Y)/F(X))$$ and  $$\log W(Q_1/Q_2)\leq \log W(F(X)/G(Y)).$$
This implies that  
\begin{align*}
\max \left\{\log W(Q_2/Q_1),\log W(Q_2/Q_1)\right\}&\leq  \max\left\{\log W(G(Y)/F(X)),\log W(G(Y)/F(X))\right\}\\
\Rightarrow d(Q_1,Q_2)&\leq d(F(X),G(Y)).
\end{align*}
Using this fact in the equation \eqref{te5}, we get
\begin{equation}\label{add 1}
	d(T_1(X),T_2(Y))\leq \frac{1}{s}d(F(X),G(Y)).
\end{equation}
It follows from the condition $(B)$ that $d(F(X),G(Y))\leq k d(X,Y)$. Using this fact with (\ref{add 1}), we get  
\begin{equation}\label{te6}
d(T_1(X),T_2(Y))\leq \frac{k}{s}d(X,Y).
\end{equation}
Let us take $\psi(t_1,t_2,t_3)=\frac{k}{s}t_1$ for all $t_1,t_2,t_3 \in \mathbb{R}_+$. Then $\psi \in \Psi_{\frac{k}{s}}$. Utilizing the definitions of $\phi$  in the equation \eqref{te6}, we get
$$d(T_1(X),T_2(Y))\leq \psi(d(X,Y),d(X,T_1(X)),d(Y,T_2(Y))).$$
So all the hypotheses of Theorem  \ref{th1} hold, and so 
 there exists $\bar{X}\in \mathcal{C}$ such that
$$T_1(\bar{X})=\bar{X}~\mbox{and}~ T_2(\bar{X})=\bar{X},$$ i.e., 
$$\bar{X}^s=Q_1 + \displaystyle \sum_{i=1}^{m} {A_i}^*F(\bar{X})A_i ~\mbox{and}~
\bar{X}^s=Q_2 + \displaystyle \sum_{i=1}^{m} {A_i}^*G(\bar{X})A_i. $$
Thus  $\bar{X}$ is a common solution of the pair of matrix equations \eqref{te1} and \eqref{te20}. Again since $\bar{X} \in \mathcal{C}$, we have $d(\bar{X},I_n) \leq a$. Moreover, since the common fixed point of $T_1$ and $T_2$ is unique, it follows that the common solution of equations \eqref{te1} and \eqref{te20} is also unique. 

Finally, by using the final part of  Theorem \ref{th1}, the given sequence $\{X_k\}$ considered in the statement  converges to the common solution $\bar{X}$.
\end{proof}

\begin{theorem} \label{tt2}
Consider the  pair of  matrix equations  as follows:
\begin{align}
X^r=\displaystyle \sum_{i=1}^{m} {A_i}^*F(X)A_i, \label{te7}\\
X^s=\displaystyle \sum_{i=1}^{m} {A_i}^*G(X)A_i \label{te8}
\end{align}
where $F,G$  are same as theorem \ref{tt1}, $A_i$'s are abritrary $n \times n$ unitary matrices, and $r,s>1$ are real  numbers. Assume that there is a real number $a\geq 0$ such that for any $X\in P(n)$ with $d(X,I_n)\leq e^{ra}$, the following conditions hold:
\begin{enumerate}
\item[$(A)$] $\lambda_{max}  \left(F(X)\right)\leq \frac{e^{ra}}{m},$ $\lambda_{\max}\left((F(X))^{-1}\right)\leq m e^{ra}$ and $\lambda_{max}  \left(G(X)\right)\leq \frac{e^{ra}}{m},$ $\lambda_{\max}\left((G(X))^{-1}\right)\leq m e^{ra}$;
\item [$(B)$]  $\lambda_{\max}\left(F(X)\right)\leq \frac{1}{m2^r} \left(\lambda_{\max}\left(Y^{-\frac{1}{2}}XY^{-\frac{1}{2}}\right)\right)^l$, $\lambda_{\max}\left(G(X)\right) \leq \frac{1}{m2^s} \left(\lambda_{\max}\left(Y^{-\frac{1}{2}}XY^{-\frac{1}{2}}\right)\right)^l$ and $\lambda_{\max}\left((F(X))^{-1}\right)$, $\lambda_{\max}\left((G(X))^{-1}\right) \leq m \left(\lambda_{\max}\left(X^{-\frac{1}{2}}YX^{-\frac{1}{2}}\right)\right)^l$, 
where $l>0$ is a real number such that $3l<\frac{rs}{r+s}.$
\end{enumerate}
Then the following assertions hold:
\begin{enumerate}
\item[$(i)$]the pair of matrix equations \eqref{te7} and \eqref{te8} posses a unique common solution $\bar{X}$ such that $\bar{X}$ is positive definite and $d(\bar{X},I_n)\leq ra$;
\item [$(ii)$] the unique solution is provided by the limit of the iterative sequence $\{X_k\}$, where 
\begin{equation}
X_k =\left\{
\begin{tabular}{ccc}
$\left(\displaystyle \sum_{i=1}^{m} {A_i}^*F(X_{k-1})A_i\right)^{\frac{1}{r}}$ & if & $k$ is odd, \\
$\left(\displaystyle \sum_{i=1}^{m} {A_i}^*G(X_{k-1})A_i\right)^{\frac{1}{s}}$ & if & $k$ is even, 
\end{tabular}
\right.
\end{equation}
and $X_0 \in P(n)$ is any element with $d(X_0,I_n) \leq ra$, and the error estimation is given by 
$$d(X_k,\bar{X}) \leq \frac{\alpha^{k-1}}{1-\alpha}d(X_0,X_1),$$ where $\alpha:=3l\left(\frac{1}{r}+\frac{1}{s}\right)$.
\end{enumerate}
\end{theorem}

\begin{proof}
Consider the set $\mathcal{C}=\{X\in P(n):d(X,I_n)\leq ra\}$ and the Thompson metric $d$ on $\mathcal{C}$. Then $\mathcal{C}$ is complete with respect to the Thompson metric $d$. 
We now consider a pair of mappings $T_1,T_2:\mathcal{C}\to P(n)$ by 
\begin{align}
T_1(X)= \left(\sum_{i=1}^{m} {A_i}^*F(X)A_i\right)^\frac{1}{r}, \label{te9}\\
T_2(X)= \left(\sum_{i=1}^{m} {A_i}^*G(X)A_i\right)^\frac{1}{s} \label{te10}
\end{align} 
for all $X\in \mathcal{C}$.  For any $X\in \mathcal{C}$, we have
\begin{align*}
d(T_1(X),I_n) &= d\left(\left(\sum_{i=1}^{m} {A_i}^*F(X)A_i\right)^\frac{1}{r},(I_n)^\frac{1}{r}\right)\\
&\leq \frac{1}{r} d\left(\left(\sum_{i=1}^{m} {A_i}^*F(X)A_i\right), I_n\right)\\
&=\frac{1}{r} d\left(\left(\sum_{i=1}^{m} {A_i}^*F(X)A_i\right), \sum_{i=1}^{m}A_i^*\left(\frac{1}{m}I_n\right){A_i}\right)\\
&\leq \frac{1}{r} d\left(F(X),\left(\frac{1}{m}I_n\right)\right).
\end{align*}
Using condition $(A)$, we get 
\begin{align*}
\lambda_{\max}\left(\left(\frac{1}{m}I_n\right)^{-\frac{1}{2}}F(X)\left(\frac{1}{m}I_n\right)^{-\frac{1}{2}}\right)&=\lambda_{\max}\left(mF(X)\right)\\
&=m\lambda_{\max}\left(F(X)\right)\\
&\leq e^{ra}\\
\Rightarrow \log \left(\lambda_{\max}\left(\left(\frac{1}{m}I_n\right)^{-\frac{1}{2}}F(X)\left(\frac{1}{m}I_n\right)^{-\frac{1}{2}}\right)\right) &\leq ra\\
\Rightarrow \log \left(W\left(F(X) \Bigg{/}\frac{1}{m}I_n\right)\right)&\leq ra,
\end{align*}
and
\begin{align*}
\lambda_{\max}\left((F(X))^{-\frac{1}{2}}\frac{1}{m}I_n(F(X))^{-\frac{1}{2}}\right)&=\lambda_{\max}\left(\frac{1}{m}(F(X))^{-1}\right)\\
&=\frac{1}{m}\lambda_{\max}\left((F(X))^{-1}\right)\\
&\leq e^{ra}\\
\Rightarrow \log\left(\lambda_{\max}\left((F(X))^{-\frac{1}{2}}\frac{1}{m}I_n(F(X))^{-\frac{1}{2}}\right)\right)&\leq ra\\
\Rightarrow \log \left(W\left(\frac{1}{m}I_n \Bigg{/}F(X)\right)\right)&\leq ra.
\end{align*}
Therefore, $d(T_1(X),I_n)\leq ra$. Similarly, we have $d(T_2(X),I_n)\leq ra$. Thus $T_1,T_2$ are self-mappings on $\mathcal{C}$.

Next, we will show that $T_1$ and $T_2$ satisfies the contractive condition in Theorem  \ref{th1}. For  $X,Y \in \mathcal{C}$, we have 
\begin{eqnarray}
	d(T_1(X),T_2(Y))
	&\leq& d\left(T_1(X),T_1(X)+\frac{1}{2}I_n\right)+d\left(T_1(X)+\frac{1}{2}I_n,T_1(X)+T_2(Y)\right)\nonumber\\
	&& +d\left(T_2(Y),T_2(Y)+\frac{1}{2}I_n\right)+d\left(T_2(Y)+\frac{1}{2}I_n,T_1(X)+T_2(Y)\right)\nonumber\\
	&\leq& d\left(T_1(X),I_n\right)+d\left(I_n,T_1(X)+\frac{1}{2}I_n\right)+d\left(T_2(Y),\frac{1}{2}I_n\right)\nonumber\\
	&&+d\left(T_2(Y),I_n\right)+d\left(I_n,T_2(Y)+\frac{1}{2}I_n\right)+d\left(T_1(X),\frac{1}{2}I_n\right)\nonumber \label{t11}\\
	&\leq& d\left(T_1(X),I_n\right)+d\left(T_2(Y),I_n\right)+2d\left(T_1(X),\frac{1}{2}I_n\right)+2d\left(T_2(Y),\frac{1}{2}I_n\right). \nonumber \\
\end{eqnarray}
Using condition $(B)$, we get 
\begin{align*}
\lambda_{\max}\left(\left(\frac{1}{m}I_n\right)^{-\frac{1}{2}}F(X)\left(\frac{1}{m}I_n\right)^{-\frac{1}{2}}\right)&=\lambda_{\max}\left(mF(X)\right)\\
&=m\lambda_{\max}\left(F(X)\right)\\
&\leq m2^r\lambda_{\max}\left(F(X)\right)\\
&\leq \left(\lambda_{\max}\left(Y^{-\frac{1}{2}}XY^{-\frac{1}{2}}\right)\right)^k\\
\Rightarrow \log \left(\lambda_{\max}\left(\left(\frac{1}{m}I_n\right)^{-\frac{1}{2}}F(X)\left(\frac{1}{m}I_n\right)^{-\frac{1}{2}}\right)\right) &\leq k \log\left(\lambda_{\max}\left(Y^{-\frac{1}{2}}XY^{-\frac{1}{2}}\right)\right)\\
\Rightarrow \log \left(W\left(F(X) \Bigg{/}\frac{1}{m}I_n\right)\right)&\leq k\log \left(W\left(X \big{/}Y\right)\right)
\end{align*}
and
\begin{align*}
\lambda_{\max}\left((F(X))^{-\frac{1}{2}}\frac{1}{m}I_n(F(X))^{-\frac{1}{2}}\right)&=\lambda_{\max}\left(\frac{1}{m}(F(X))^{-1}\right)\\
&=\frac{1}{m}\lambda_{\max}\left((F(X))^{-1}\right)\\
&\leq \left(\lambda_{\max}\left(X^{-\frac{1}{2}}YX^{-\frac{1}{2}}\right)\right)^k\\
\Rightarrow \log\left(\lambda_{\max}\left((F(X))^{-\frac{1}{2}}\frac{1}{m}I_n(F(X))^{-\frac{1}{2}}\right)\right)&\leq k \log\left(\lambda_{\max}\left(X^{-\frac{1}{2}}YX^{-\frac{1}{2}}\right)\right)\\
\Rightarrow \log \left(W\left(\frac{1}{m}I_n \Bigg{/}F(X)\right)\right)&\leq k \log\left(W\left(Y \big{/}X\right)\right).
\end{align*}
Therefore, 
{\small
\begin{align*}
\max\left\{\log \left(W\left(F(X) \Bigg{/}\frac{1}{m}I_n\right)\right),\log \left(W\left(\frac{1}{m}I_n \Bigg{/}F(X)\right)\right)\right\}&\leq k\max\left\{\log \left(W\left(X \big{/}Y\right)\right),\log \left(W\left(Y \big{/}X\right)\right)\right\}\\
\Rightarrow d\left(F(X),\frac{1}{m}I_n\right)&\leq k d(X,Y)\\
\Rightarrow d(T_1(X),I_n)&\leq \frac{k}{r}d(X,Y).
\end{align*}
}
Similarly, we can show that $d(T_2(Y),I_n)\leq \frac{k}{s}d(X,Y).$ Again 
\begin{align*}
d\left(T_1(X),\frac{1}{2}I_n\right)&=d\left(\left(\sum_{i=1}^{m} {A_i}^*F(X)A_i\right)^\frac{1}{r},\left(\frac{1}{2^r}I_n\right)^\frac{1}{r}\right)\\
&\leq \frac{1}{r} d\left(\left(\sum_{i=1}^{m} {A_i}^*F(X)A_i\right),\left(\frac{1}{2^r}I_n\right)\right)\\
&=\frac{1}{r} d\left(\left(\sum_{i=1}^{m} {A_i}^*F(X)A_i\right),\left(\sum_{i=1}^{m}A_i^*\left(\frac{1}{m 2^r}I_n\right)A_i\right)\right)\\
&\leq \frac{1}{r} d\left(F(X),\frac{1}{m 2^r}I_n\right).
\end{align*} 
Using the condition $(B)$,we get
\begin{align*}
\lambda_{\max}\left(\left(\frac{1}{m2^r}I_n\right)^{-\frac{1}{2}}F(X)\left(\frac{1}{m2^r}I_n\right)^{-\frac{1}{2}}\right)&=\lambda_{\max}\left(m2^rF(X)\right)\\
&=m2^r \lambda_{\max}(F(X))\\
&\leq \left(\lambda_{\max}\left(Y^{-\frac{1}{2}}XY^{-\frac{1}{2}}\right)\right)^k\\
\Longrightarrow \log \left(\lambda_{\max}\left(\left(\frac{1}{m2^r}I_n\right)^{-\frac{1}{2}}F(X)\left(\frac{1}{m2^r}I_n\right)^{-\frac{1}{2}}\right)\right) &\leq k \log\left(\lambda_{\max}\left(Y^{-\frac{1}{2}}XY^{-\frac{1}{2}}\right)\right)\\
\Longrightarrow \log \left(W\left(F(X) \Bigg{/}\frac{1}{m2^r}I_n\right)\right)&\leq k\log \left(W\left(X \big{/}Y\right)\right),
\end{align*}
and
\begin{align*}
\lambda_{\max}\left(\left(F(X)\right)^{-\frac{1}{2}}\frac{1}{m2^r}I_n\left(F(X)\right)^{-\frac{1}{2}}\right)&=\lambda_{\max}\left(\frac{1}{m2^r}\left(F(X)\right)^{-1}\right)\\
&=\frac{1}{m2^r}\lambda_{\max}\left(\left(F(X)\right)^{-1}\right)\\
&\leq \frac{1}{m}\lambda_{\max}\left(\left(F(X)\right)^{-1}\right)\\
&\leq \left(\lambda_{\max}\left(X^{-\frac{1}{2}}YX^{-\frac{1}{2}}\right)\right)^k\\
\Rightarrow \log\left(\lambda_{\max}\left(\left(F(X)\right)^{-\frac{1}{2}}\frac{1}{m2^r}I_n\left(F(X)\right)^{-\frac{1}{2}}\right)\right)&\leq k \log\left(\lambda_{\max}\left(X^{-\frac{1}{2}}YX^{-\frac{1}{2}}\right)\right)\\
\Rightarrow \log\left(W\left(\frac{1}{m2^r}I_n\Bigg{/}F(X)\right)\right)&\leq k \log\left(Y/X\right).
\end{align*}
Therefore,
{\small
\begin{align*}
\max\left\{\log \left(W\left(F(X) \Bigg{/}\frac{1}{m2^r}I_n\right)\right),\log\left(W\left(\frac{1}{m2^r}I_n\Bigg{/}F(X)\right)\right)\right\}&\leq k\max\left\{\log \left(W\left(X \big{/}Y\right)\right),\log \left(W\left(Y \big{/}X\right)\right)\right\}\\
\Longrightarrow d\left(F(X),\frac{1}{m2^r}I_n\right)&\leq kd(X,Y)\\
\Longrightarrow d\left(T_1(X),\frac{1}{2}I_n\right)&\leq \frac{k}{r}d(X,Y).
\end{align*}
}
Similarly, we can show that $d\left(T_2(Y),\frac{1}{2}I_n\right)\leq \frac{k}{s}d(X,Y).$  Utilizing all these calculations in equation \eqref{t11}, we deduce  
$$d(T_1(X),T_2(Y))\leq \alpha d(X,Y),$$
where $\alpha=3k\left(\frac{1}{r}+\frac{1}{s}\right)<1$.

Let us consider $\psi(t_1,t_2,t_3)=\alpha t_1$ for all $t_1,t_2,t_3 \in \mathbb{R}_+$. Then $\psi\in \Psi_{\alpha}$ and 
$$ d(T_1(X),T_2(Y)) \leq \psi\left(d(X,Y),d(X,T_1(X)),d(Y,T_2(Y))\right)$$ for all $X,Y \in \mathcal{C}$. The rest of the proof is similar to that of Theorem \ref{tt1} and so omitted.
\end{proof}
\section{Numerical examples}
In this section, we validate the results obtained in the previous section by some examples and showing 
the numerical approximations of convergence of iterated sequences using  diagrams.
\begin{example}
Consider the system of matrix equations \eqref{te1} and \eqref{te20} for $s=2$, $m=1$, $F(X)=X^{\frac{1}{2}}$ and $G(X)=X^{\frac{1}{3}}$, that is,
\begin{align}
X^2=Q_1&+A_1^*X^{\frac{1}{2}}A_1 \label{ne1}\\
X^2=Q_2&+A_1^*X^{\frac{1}{3}}A_1 \label{ne2},
\end{align}
where we choose \( A_1=
  \begin{pmatrix}
    4+5i & 2+7i & -3+2i  \\
    5-6i & 4-3i & -i+8\\
    9-i & 5+2i & 6-2i 
  \end{pmatrix},\ 
Q_1=
  \begin{pmatrix}
    2 & -1 & 0  \\
    - 1 & 2 & -1\\
    0 & -1 & 2
  \end{pmatrix} \mbox{and }
  \) \\ 
\(Q_2=
  \begin{pmatrix}
    6 & 0 & 0 \\
    0 & 6 & 0\\
    0 & 0 & 6
  \end{pmatrix}.
\)

Then if we choose $a=10$, then using MATLAB, we check that  for any $X\in P(3)$ with $d(X,I_3)\leq e^{10}$, the assumptions $(A)-(C)$ of Theorem \ref{tt1} are fulfilled. So by the consequences of this theorem, matrix equations \eqref{ne1} and \eqref{ne2} posses a unique common solution $\bar{X}$ such that $\bar{X}$ is positive definite and $d(\bar{X},I_3)\leq 10$.
 Indeed, we see that the unique positive definite solution is
\[ \bar{X}=
  \begin{pmatrix}
    5.6933  & 2.4413+1.3428i & 1.7040+0.5152i\\
 2.4413-1.3428i & 4.4438 &  0.7308+0.3634i\\
 1.7040-0.5152i & 0.7308-0.3634i & 5.0193
  \end{pmatrix}.
\]
Note that $d(\bar{X},I_3)=2.17614\leq 10$. The graphical representation of the convergence history  of the sequence $\{X_n\}$ considered in Theorem \ref{tt1} to the unique solution by taking two distinct initial values is shown in Figure 1.
\begin{figure}[h!]
\begin{center}
\textbf{Convergence behaviour}
\includegraphics[scale=1.00]{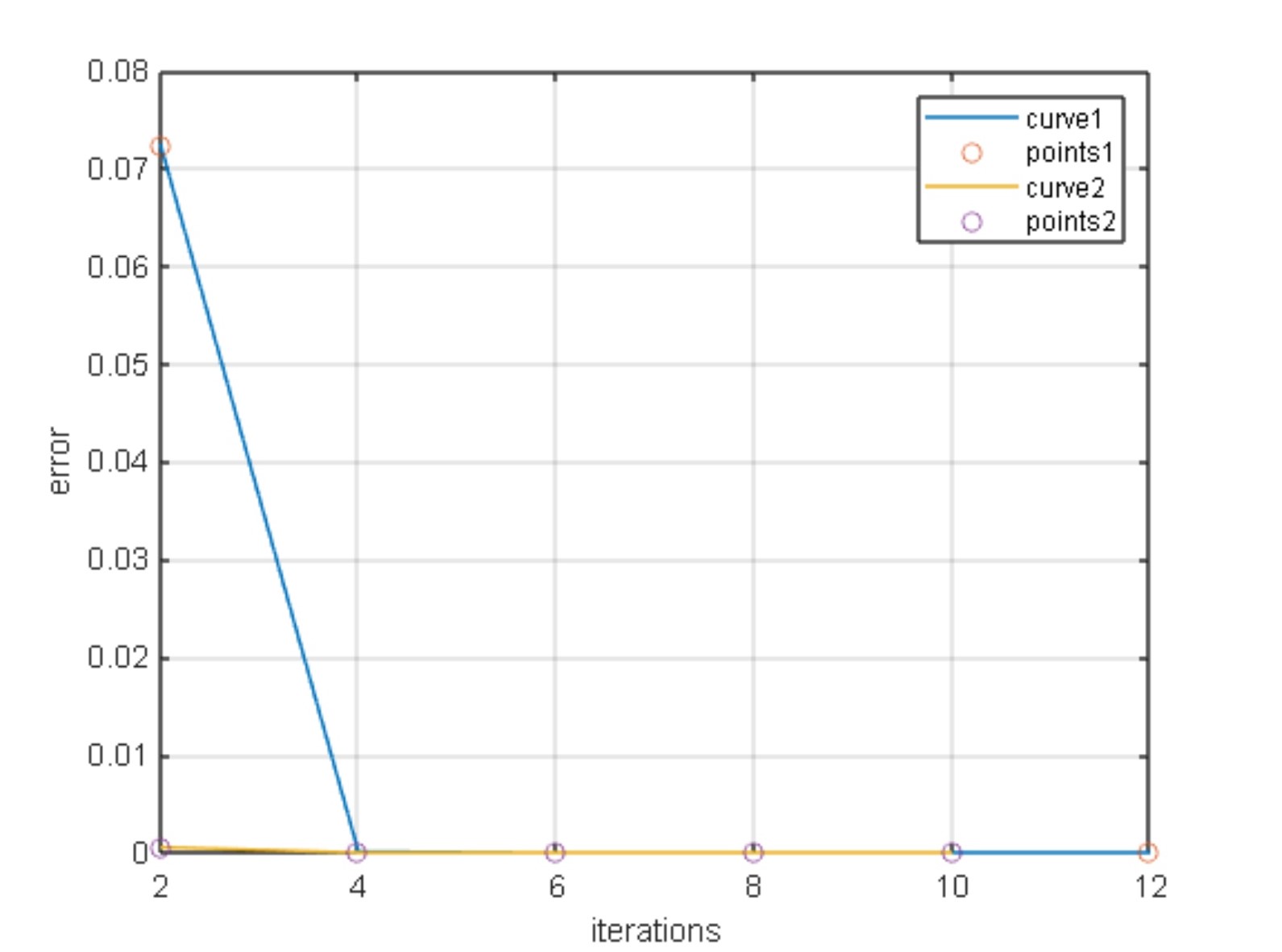}
\textbf{Figure 1}

\end{center}
\end{figure}
\end{example}
\begin{example}
We consider the system of matrix equations \eqref{te7} and \eqref{te8} for $r=2$, $s=3$, $m=1$, $F(X)=X^{\frac{1}{2}}$ and $G(X)=X^{\frac{1}{4}}$, that is, we consider the following system of equations
\begin{align}
X^2=A_1^*X^{\frac{1}{2}}A_1\label{ne3}\\
X^3=A_1^*X^{\frac{1}{4}}A_1 \label{ne4},
\end{align}
where we choose
\[ A_1=
  \frac{1}{3}\begin{pmatrix}
    2 & -2 & 1  \\
    1 & 2 & 2\\
    2 & 1 & -2 
  \end{pmatrix}.\]
 Then for $a=2$, using MATLAB, we check that for any $X\in P(3)$ with $d(X,I_3)\leq e^{2\cdot 2}=e^4$, the assumptions $(A),\ (B)$ of Theorem \ref{tt2} are fulfilled.  So by this theorem, matrix equations \eqref{ne3}and \eqref{ne4} posses a unique common solution $\bar{X}$ such that $\bar{X}$ is positive definite and $d(\bar{X},I_3)\leq 2$.  We obtain by calculating  that the unique common positive definite solution $\bar{X}$ is 
\[ \bar{X}=
  \begin{pmatrix}
    1  & 0 & 0\\
 0 & 1 &  0\\
0 & 0 & 1
  \end{pmatrix}.
\]
Note that $d(\bar{X},I_3)=0\leq 2$. The graphical representation of the convergence history  of the sequence $\{X_n\}$ considered in Theorem \ref{tt2} to the unique solution by taking two distinct initial values is shown in Figure 2.
 \begin{figure}[h!]
\begin{center}
\textbf{Convergence behaviour}\\
\includegraphics[scale=2.50]{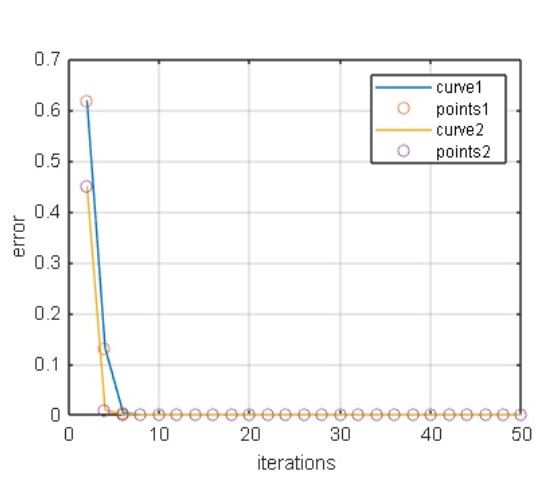}\\
\textbf{Figure 2}

\end{center}
\end{figure}
\end{example}
\begin{Acknowledgement}
The first and second authors are  thankful to CSIR, Government of India for their financial support (Ref. No. 09/973 $(0018)/2017$-EMR-I and $25(0285)/18/$EMR-II).
\end{Acknowledgement}


\end{document}